\documentclass[a4paper,11pt]{amsart}

\usepackage{verbatim}
\usepackage{microtype}
\usepackage{amsmath,amsthm,amssymb}
\usepackage{bbold}
\usepackage[all]{xy}
\usepackage{xspace}
\usepackage{mathrsfs}
\usepackage{enumerate}
\usepackage{todonotes}
\usepackage{color}
  \definecolor{mycolor}{rgb}{0.2,0.3,1}

\usepackage[hyperindex,backref,pagebackref,pdfborder={0 0
  0},citecolor=mycolor,linkcolor=mycolor,colorlinks=true 
]{hyperref}

\newcommand{\C}{\ensuremath{\mathbb{C}}\xspace}

\newcommand{\ma}{\ensuremath{\mathcal{A}}\xspace}
\newcommand{\mb}{\ensuremath{\mathcal{B}}\xspace}

\newcommand{\mc}{\ensuremath{\mathcal{C}}\xspace}

\newcommand{\pr}[2]{\ensuremath{\langle {#1},{#2}\rangle}}
\newcommand{\norm}[1]{\ensuremath{\|#1\|}}

\newcommand{\Mb}{\mathbb{M}}
\newcommand{\Nb}{\mathbb{N}}

\newcommand{\tb}{\mathbb{t}}

\DeclareMathOperator{\supp}{supp}

\theoremstyle{plain}
\newtheorem{thm}{Theorem}[section]
\newtheorem{prop}[thm]{Proposition}
\newtheorem{lem}[thm]{Lemma}
\newtheorem{cor}[thm]{Corollary}          

\theoremstyle{definition}
\newtheorem{defn}[thm]{Definition}
\newtheorem{exmp}[thm]{Example}

\theoremstyle{remark} 
\newtheorem{rem}[thm]{Remark}

\newcounter{horas}\newcounter{minutos}
\newcommand{\printtime}{%
\setcounter{horas}{\time/60}
\setcounter{minutos}{\time-\value{horas}*60}%
\ifthenelse{\value{horas}<10}{0}{}\thehoras:%
\ifthenelse{\value{minutos}<10}{0}{}\theminutos\ %
}

\begin{document}

\title[Conditional expectations on Fell bundles]{Conditional expectations on Fell bundles\date{October 2021}}
\author{Fernando Abadie}
\address{CMAT, FC, Universidad de la
    Rep\'ublica. Igu\'a~4225, 11400, Montevideo,
    URUGUAY}
\email{fabadie@cmat.edu.uy}

\maketitle

\begin{abstract}
We show that the existence of a continuous conditional expectation
from a Fell bundle to a Fell subbundle implies that the full
cross-sectional C*-algebra of the subbundle is contained in the full
cross-sectional C*-algebra of the bundle, and moreover there exists a
conditional expectation from the latter algebra to the former.
A similar fact holds for the reduced cross-sectional C*-algebras. Besides, the subbundle is amenable whenever so is the bundle.      
\end{abstract}
\textbf{\textcolor{mycolor}{\tableofcontents}}

\section{Introduction and preliminaries}

\par A very important concept in probability theory is that of conditional expectation. Adapted to the “non-commutative” realm of operator algebras,
here too it became an important tool, first for von Neumann algebras and then for C* algebras. In this case, the conditional expectation is a certain type of linear operator from an algebra to a subalgebra (we refer to \cite{kad} for a nice description and applications of this passage from the classical probability theory to the world of operator algebras).
\par If $A$ is a C*-subalgebra of a C*-algebra $B$, a conditional expectation from $B$ to $A$ is a linear map $E:B\to A$ such that $E^2=E$ (i.e.: $E$ is an idempotent), $E(B)=A$, $E(aba')=aE(b)a'$ for $a,a'\in A$, $b\in B$, and moreover $E$ is completely contractive and completely positive (ccp). However, thanks to Tomiyama's theorem \cite[Theorem~1.5.10]{bo}, they can be more easily described simply as contractive idempotents whose range is a C*-subalgebra. 
\par Conditional expectations have been studied in particular in relation to one of the most important constructions in operator algebras, that of crossed products of C*-algebras by groups (\cite{ad1}, \cite{ad2}, \cite{i}, \cite{ired}, \cite{mk}).
\par In addition to crossed products, there are a variety of other constructions that give rise to graded algebras over groups that are equally useful, and for which it would be nice to enjoy similar results as for crossed products. Fortunately, such algebras can often be described as sectional algebras of Fell bundles. 
\par Thus in this paper we study conditional expectations on Fell bundles. In order to describe briefly what we will do, we will use some terminology and notation that we will only introduce towards the end of this section, and to which we refer the reader.  
\par If $\ma$ is a Fell subbundle of a Fell bundle $\mb$, it is not in general true that $C^*(\ma)$ is a C*-subalgebra of $C^*(\mb)$. Some recent examples of this situation can be found in \cite{suzuki} in the realm of crossed prodcuts, and even for Fell bundles over discrete groups, as shown in \cite[Proposition~21.7]{examen}. On the other hand, there are some cases in which we know that $C^*(\ma)$ is a C*-subalgebra of $C^*(\mb)$, for instance when $\ma$ is an hereditary Fell subbundle of $\mb$ (see \cite[Theorem~4.3]{equivfb} and \cite[Corollary~5.4]{trings}). Further back in time we have the following result from Itoh in relation to crossed products \cite{i} (see also \cite{mk} for the case of discrete groups): suppose $\alpha$ is an action of the locally compact group $G$ on the C*-algebra $B$, $A$ is a $\alpha$-invariant C*-subalgebra of $B$ and $E:B\to A$ is a conditional expectation from $B$ to $A$ such that $\alpha_tE=E\alpha_t$, $\forall t\in G$. Then $A\rtimes_{\alpha|_{A}}G$ is a C*-subalgebra of $B\rtimes_\alpha G$, and $E$ can be extended to a conditional expectation $E^u:B\rtimes_\alpha G\to A\rtimes_\alpha G$. This can be translated to the language of Fell bundles as follows. Let $\mb=G\times B$ and $\ma=G\times A$ be the semi-direct product Fell bundles associated to the actions $\alpha$ and $\alpha|_{A}$ respectively, so that $\ma$ is a Fell subbundle of $\mb$. The conditional expectation $E:B\to A$ extends to a map (that we also call) $E:\mb\to\ma$ by means of $E(r,b):=(r,E(b))$, $\forall (r,b)\in\mb$. It is easy to see that this new $E$ is a conditional expectation from the Fell bundle $\mb$ to its Fell subbundle $\ma$ according to Definition~\ref{defn:condexp} below. Then Itoh's result says that, in this case, there is an inclusion $C^*(\ma)\subseteq C^*(\mb)$, and that the conditional expectation $E$ defines a conditional expectation $E^u:C^*(\mb)\to C^*(\ma)$. This is exactly the first main result we want to prove (Theorem~\ref{thm:subalgebra}), but with the only assumption that $\ma$ is a Fell subbundle of $\mb$ and the existence of a conditional expectation between them.
\par On the other hand, one can consider the same situation, but replacing the full cross-sectional C*-algebras of the Fell bundles by the reduced cross-sectional ones. In this case there is no need to worry about the inclusion of $C^*_r(\ma)$ in $C^*_r(\mb)$, which always holds (\cite[Proposition~3.2]{fa}, and \cite[Proposition~21.3]{ruybook} for a simple proof in the discrete group case), but it is worth asking whether there is also a conditional expectation $E^r:C^*_r(\mb)\to C^*_r(\ma)$ between them. Again, this situation has been already considered by Itoh in \cite{ired} and Khoshkam in \cite{mk} for actions of general locally compact groups and of discrete groups respectively, and by Exel for Fell bundles over discrete groups (\cite[Theorem~21.29]{ruybook}). We will also show (Theorem~\ref{thm:subalgebrared}) that the above-mentioned results of Itoh, Khoshkam and Exel generalize for any conditional expectation $E:\mb\to \ma$ between Fell bundles over arbitrary locally compact groups.

\medskip
 \par  Although we assume that the reader is familiar with the basic definitions and facts surrounding Fell bundles over locally compact groups, to fix the notation and terminology we recall below some of the objects we will be working with. The reader is referred to the treatise \cite{fd} for complete information on this topic. 
 \par If $\mb=(B_t)_{t\in G}$ is a Fell bundle over the locally compact group $G$, we denote by $C_c(\mb)$ the vector space of compactly supported continuous sections of $\mb$. It is a locally convex topological space with the inductive limit topology determined by the family of subspaces $C_K(\mb):=\{f\in C_c(\mb):\supp(f)\subseteq K\}$, $K\subseteq G$ compact.  Besides, $C_c(\mb)$ is a $*$-algebra (``the compacted $*$-algebra of $\mb$''), with the product given by convolution: $f*g(t):=\int_Gf(s)g(s^{-1}t)ds$ (integration with respect to the Haar measure of $G$) and involution given by $f^*(t):=\Delta^{-1}f(t^{-1})^*$, where $\Delta$ is the modular function of $G$. Moreover, $\norm{f}_1 :=\int_G\norm{f(t)}dt$ is a norm on $C_c(\mb)$, with which it becomes a normed $*$-algebra. Its completion, $L^1(\mb)$, is a Banach $*$-algebra whose non-degenerate representations correspond to the non-degenerate representations of $\mb$ (integration of representations of $\mb$ gives representations of $L^1(\mb)$). Its universal enveloping C*-algebra, denoted $C^*(\mb)$, is called the full cross-sectional algebra of $\mb$. Consequently we have a bijective correspondence between non-degenerate representations of $\mb$ and non-degenerate representations of $C^*(\mb)$ (see \cite[VIII]{fd}) for the precise statements.
 \par Among the representations of $\mb$, there is one of special importance, the left regular representation. The space $C_c(\mb)$ is also a right $B_e$-module with $(fb)(t):=f(t)b$, $\forall f\in C_c(\mb)$, $b\in B_e$ and $t\in G$ (here $e$ denotes the identity of the group $G$). Moreover, the map $C_c(\mb)\times C_c(\mb)\ni(f,g)\mapsto \int_Gf(r)^*g(r)dr\in B_e$ is a pre-inner product on $C_c(\mb)$, so by completing we obtain a full right Hilbert $B_e$-module $\big(L^2(\mb),\pr{\,}{}_{L^2(\mb)}\big)$. Given $b_t\in B_t$ and $\xi\in C_c(\mb)$, we have that $\Lambda^\mb_{b_t}(\xi):G\to \mb$, given by $\Lambda^\mb_{b_t}(\xi)(r):=b_t\xi(t^{-1}r)$ is a continuous section of $\mb$, so it belongs also to $L^2(\mb)$. It can be shown that, in fact, the map $\xi\mapsto \Lambda^\mb_{b_t}\xi$ can be extended to an adjointable map in $L^2(\mb)$, and it turns out that $b\mapsto \Lambda^\mb_{b}$ is a representation of $\mb$: the left regular representation of $\mb$
The image of $C^*(\mb)$ under the (integrated form of the) representation $\Lambda^\mb$ is called the reduced cross-sectional algebra of $\mb$, and it is denoted by $C^*_r(\mb)$. Of course $C^*_r(\mb)\subseteq \mathcal{L}(L^2(\mb))$, the C*-algebra of adjointable operators on $L^2(\mb)$. When $\Lambda^\mb$ is injective, so we can identify $C^*(\mb)$ and $C^*_r(\mb)$, it is said that the Fell bundle $\mb$ is amenable.   
\medskip
\par The organization of the article is as follows. In the next section, after defining and studying some basic properties of conditional expectations, we study some maps on certain auxiliary C*-algebras $\Mb_\tb(\mb)$ associated with the bundles, which will prove helpful to prove, in the third section, our first main  result. There we extend the conditional expectation $E$ to a map $E^c:C_c(\mb)\to C_c(\ma)$ such that $E^c(f)(t)=E(f(t))$, we show that we have an inclusion $C^*(\ma)\subseteq C^*(\mb)$, and that the map $E^c$ extend to a completely positive and completely contractive map $E^r:C^*(\mb)\subseteq C^*(\ma)$. To do it we use the results of the second section, and, crucially, the correspondence between functionals of positive type and cyclic representations of a Fell bundle. Then, using that the norm of a completely positive map can be computed using approximate units, we see that $E^r$ is in fact completely contractive, and therefore a conditional expectation. As an application we show that if $\mb$ is amenable, then so is $\ma$. We also give an application concerning the C*-algebras $\Mb_\tb(\mb)$. In the last section we prove the second main result. For this we make a heavy use of a concrete way to see $C^*(\ma)$ inside $C^*(\mb)$, as well as the possibility of creating a useful Hilbert module $L^2_E$ using the conditional expectation $E$, which allows to conveniently combine the regular representations of $\ma$ and $\mb$.       

\section{Conditional expectations on Fell bundles}
 \begin{defn}\label{defn:condexp} Let $\mb=(B_t)_{t\in G}$ be a Fell bundle over the locally compact group $G$, and let
   $\ma=(A_t)_{t\in G}$ be a Fell subbundle of $\mb$. A continuous idempotent quasi-linear map 
   $E:\mb\to\ma$ is said to be a conditional expectation from $\mb$ to
   $\ma$ if $E(\mb)=\ma$, $\norm{E|_{B_e}}=1$ and $E(aba')=aE(b)a'$,
   $\forall b\in\mb$, $a,a'\in \ma$.  \par Recall that $E$ is said to be quasi-linear when $E:B_t\to A_t$ is linear for each $t\in G$.
 \end{defn}
\par It is clear that if $E$ is such a conditional expectation, then
$E(a)=a$ $\forall a\in \ma$, and also that the restriction of $E$ to
$B_e$ is a conditional expectation from $B_e$ to $A_e$ in the sense of C*-algebras.
\par As we shall see in Corollary~\ref{cor:condexp}, the definition above implies that $\norm{E|_{B_t}}\leq 1$, $\forall t\in G$. In view of Tomiyama's theorem one may ask whether the bimodule condition is strictly necessary or can be replaced by the requirement that $E$ be contractive in each fiber, but we have not attempted to answer this question.    

In our
definition of conditional expectation we do not require condition (i)
nor the boundedness of the maps $P_g$ of
\cite[Definition~21.19]{ruybook}. However, as it will follows   
from Lemma~\ref{lem:condexp} and Corollary~\ref{cor:condexp} below,
such requirements follow 
automatically, so both definitions are actually equivalent.
\begin{exmp}[Canonical expectations]\label{exmp:canexp}
An important role in the theory of Fell bundles over discrete groups, in particular in relation to the question of amenability
(see \cite{examen}), has been played by the canonical expectation
$E:\mb\to B_e$ given by $E(b)=\begin{cases} b&\textrm{ if }b\in 
  B_e\\ 0&\textrm{ if }b\notin B_e\end{cases}$ (we look at $B_e$ as a Fell bundle over the trivial group $\{e\}$). 
\par More generally, given a subgroup $H$ of $G$, one can define
$E_H:\mb\to \mb_H$ by $E_H(b_t)=\begin{cases} b_t&\textrm{ if }t\in 
  H\\ 0&\textrm{ if }t\notin H\end{cases}$, $\forall b_t\in B_t$. 
Here $\mb_H$ is the retraction of $\mb$ to $H$.
\end{exmp}

\begin{exmp}[Inflations]\label{exmp:inflations}
Given a Fell bundle $\mb$, we can form the Fell bundle $M_n(\mb)$,
such that $M_n(\mb)_t:=M_n(B_t)$, $\forall t\in G$. Moreover, if
$E:\mb\to\ma$ is a conditional expectation from $\mb$ to $\ma$, the
inflation $E_n:M_n(\mb)\to M_n(\ma)$, such that
$E_n(b_{ij}):=(E(b_{ij}))$, is clearly a conditional expectation from
$M_n(\mb)$ to $M_n(\ma)$. Alternatively one can see $M_n(\mb)$ as $M_n\otimes \mb$ in the sense of \cite{fatens}, so $E_n$ becomes $Id\otimes E$.  
\end{exmp}

Since the restriction $E|_{B_e}:B_e\to A_e$ is a conditional
expectation, it satisfies $E(b^*)=E(b)^*$ and $E(b^*b)\geq
 E(b)^*E(b)$ $\forall b\in B_e$. In fact we can remove the condition 
 $b\in B_e$: 

 \begin{lem}\label{lem:condexp}
 Let $E:\mb\to\ma$ be a conditional expectation from $\mb$ to the Fell
 subbundle $\ma$ of $\mb$. Then $E(b^*)=E(b)^*$ and $E(b^*b)\geq
 E(b)^*E(b)$, $\forall b\in\mb$.   
 \end{lem}
\begin{proof}
Let $b\in B_t$ and $a\in A_t$. Since $a^*b\in B_e$ and the restriction
of $E$ to $B_e$ is a conditional expectation, we have
\[\big(E(b^*)-E(b)^*\big)a=E(b^*a)-E(a^*b)^*=E(b^*a)-E((a^*b)^*)=0.\] Thus 
$E(b^*)-E(b)^*=0$, for one can take $a=(E(b^*)-E(b)^*)^*$. The last 
statement follows from the first one: 
\begin{gather*}0\leq E\big((b^*-E(b)^*)(b-E(b))\big)
=E\big(b^*b-b^*E(b)-E(b)^*b+E(b)^*E(b)\big)\\
=E(b^*b)-E\big(b^*E(b)\big)-E\big(E(b)^*b\big)+E\big(E(b)^*E(b)\big)
=E(b^*b)-E(b)^*E(b). 
\end{gather*}  
\end{proof}

\begin{cor}\label{cor:condexp}
$E$ is contractive: $\norm{E(b)}\leq\norm{b}$, $\forall b\in\mb$.
\end{cor}
\begin{proof}
Since $0\leq E(b)^*E(b)\leq E(b^*b)$, we have
$\norm{E(b)}^2=\norm{E(b)^*E(b)}\leq \norm{E(b^*b)}\leq \norm{b}^2$,
where the latter inequality follows from the fact that $E|_{B_e}$ is contractive by definition. 
\end{proof}

\par We turn to prove our first main result. For this we will use some auxiliary algebras that are linked to a Fell bundle, and that have already proved their usefulness in other occasions (see for instance \cite{equivfb}).   
\par Given $\tb=(t_1,\ldots,t_n)\in G^n$, define   
\[\Mb_\tb(\mb):=\{ M\in
M_n(\mb)\colon M_{i,j}\in B_{t_i{t_j}^{-1}}\ \forall\
i,j=1,\ldots,n\},\]
\[\Nb_\tb(\mb):=\{ N\in
M_n(\mb)\colon N_{i,j}\in B_{t_i}B_{t_j^{-1}}\ \forall\
i,j=1,\ldots,n\},\]

and 
\[Y_\tb:=B_{t_1}\oplus\cdots \oplus B_{t_n},\]
where the latter is a direct sum of right Hilbert $B_e$-modules. 
\par It is easy to see that, with the usual linear operations, product and
involution of matrices, $\Mb_\tb(\mb)$ is a $*$-algebra, and that
$Y_{\tb }$ is a left $\Mb_\tb(\mb)$-module (we think of the elements of $Y_\tb$ as column matrices). For instance, if $M,N\in\Mb_\tb(\mb)$, then
\[(MN)_{i,j}
=\sum_{k}M_{i,k}N_{k,j}
\in\sum_kB_{t_it_k^{-1}}B_{t_kt_j^{-1}}
\subseteq B_{t_it_j^{-1}},\] and a similar computation shows that
$\Mb_\tb(\mb)Y_{\tb}\subseteq Y_{\tb}$. Besides, $\Nb_\tb(\mb)$ is clearly a closed $*$-ideal of $\Mb_\tb(\mb)$.
Note that, in fact, we have:

\begin{gather}\label{eqn:ytbim}
  \mathcal{K}(Y_\tb)=\mathcal{K}(\oplus_{i=1}^nB_{t_i})=(\mathcal{K}(B_{t_i},B_{t_j})_{i,j=1}^n=(B_{t_i}B_{t_j}^*)_{i,j=1}^n=\Nb_\tb(\mb).    
\end{gather}
On the other hand $Y_\tb$ is a full right Hilbert $J_\tb(\mb)$-module over the ideal $J_\tb(\mb)$ of $B_e$, where $J_\tb(\mb):=\sum_{i=1}^nB_{t_i}^*B_{t_i}$. Therefore $Y_\tb$ is an equivalence bimodule between $\Nb_\tb(\mb)$ and $J_\tb(\mb)$. We denote by $\pr{\,}{}_l$ and  $\pr{\,}{}_r$ the corresponding left and right inner products of $Y_\tb$.    
\par For $\tb=(t_1,\ldots,t_n)\in G^n$ and $r\in G$, we put $\tb r=(t_1r,\ldots,t_nr)\in G^n$. Note that $\Mb_{\tb r}(\mb)=\Mb_{\tb
}(\mb)$, $\forall r\in G$, and  that $\Mb_\tb(\mb)Y_{\tb r}\subseteq Y_{\tb r}$, and $\Nb_{\tb r}(\mb)=(B_{t_ir}B_{t_j r}^*)_{i,j=1}^n$ is also an ideal in $\Mb_\tb(\mb)$. 

\begin{prop}\label{prop:compactYbim}
  Let $\tb=(t_1,\ldots,t_n)\in G^n$. Then
  \[\Mb_\tb(\mb)=\textrm{span}\{\Nb_{\tb t_k^{-1}}(\mb):k=1,\ldots,n\}
=\textrm{span}\{\Nb_{\tb r^{-1}}(\mb):r\in G\}
    .\] 
\end{prop}
\begin{proof}
  The preceding comments imply $\textrm{span}\{\Nb_{\tb t_k^{-1}}(\mb):k=1,\ldots,n\}\subseteq \textrm{span}\{\Nb_{\tb r^{-1}}(\mb):r\in G\}\subseteq \Mb_\tb(\mb)$. On the other hand, by \eqref{eqn:ytbim} the $(i,j)$ component of $\textrm{span}\{\Nb_{\tb t_k^{-1}}(\mb):j=1,\ldots,n\}$ is the set $\sum_{m=1}^nB_{t_it_m^{-1}}B_{t_j t_m^{-1}}^*$, which clearly contains $B_{t_it_j^{-1}}$, the $(i,j)$ component of $\Mb_\tb(\mb)$. This shows that $\Mb_\tb(\mb)\subseteq\textrm{span}\{\Nb_{\tb t_k^{-1}}(\mb):k=1,\ldots,n\}$, which ends the proof. 
\end{proof}

\begin{cor}\label{cor:Mpos}
Let $\tb=(t_1,\ldots,t_n)\in G^n$, and $M\in \Mb_\tb(\mb)$. Then $M$
is a positive element of the C*-algebra $\Mb_\tb(\mb)$ if and only if 
$\forall k=1,\ldots,n$ and $\forall y\in Y_{\tb t_k^{-1}}$ we have
$\pr{My}{y}_r\geq 0$. More explicitly $M=(M_{ij})\in \Mb_\tb(\mb)^+\iff \sum_{i,j=1}^ny_i^*M_{ij}y_j\in B_e^+$ $\forall (y_1,\ldots,y_n)\in Y_{\tb t_k^{-1}}$ and $k=1,\ldots,n$.    
\end{cor}
\begin{proof}
  It is clear that the positivity of $M$ implies that $\pr{My}{y}_r\geq 0$ $\forall y\in Y_{\tb t_k^{-1}}$ and $\forall k=1,\ldots,n$. To prove the converse, suppose that $M$ satisfies this latter condition.    
Since $Y_{\tb t_k^{-1}}$ is a Hilbert $(\Mb_\tb(\mb)-B_e)$-bimodule, we have
a homomorphism $\mu_{t_k^{-1}}:\Mb_\tb(\mb)\to\mathcal{L}(Y_{\tb {t_k^{-1}}})$, given by $\mu(N)y:=Ny$. Note that $\mu$ is injective when restricted to each $\Nb_{\tb t_k^{-1}}(\mb)$, and in fact above we identified $\Nb_{\tb t_k^{-1}}(\mb)$ with $\mathcal{K}(Y_{\tb t_k^{-1}})$ under this homomorphism. According to \cite[Lemma~4.1]{l}, our assumption implies that $\mu_{t_k^{-1}}(M)$ is a positive element of $\mathcal{L}(Y_{\tb {t_k^{-1}}})$, $\forall k=1,\ldots n$.  
Define $\mu:\Mb_\tb(\mb)\to\oplus_{k=1}^n\mathcal{L}(Y_{\tb t_k^{-1}})$ by $\mu(N):=(\mu_{t_k^{-1}}(N))_{k=1,\ldots,n}$, $\forall N\in \Mb_\tb(\mb)$. Then $\mu(M)$ is positive in the C*-algebra $\oplus_{k=1}^n\mathcal{L}(Y_{\tb t_k^{-1}})$, and to see that $M$ is positive it is enough to prove that $\mu$ is injective. 
So let $N\in\ker\mu$. Since $\Nb_{\tb t_k^{-1}}(\mb)$ is an ideal in $\Mb_{\tb }(\mb)$, we have $N \Nb_{\tb t_k^{-1}}(\mb)\subseteq \Nb_{\tb t_k^{-1}}(\mb)$. But $\mu\big(N \Nb_{\tb t_k^{-1}}(\mb)\big)=\mu(N)\mu(\Nb_{\tb t_k^{-1}}(\mb))=0$ and $\mu$ is faithful on $\Nb_{\tb t_k^{-1}}(\mb)$, which implies $N \Nb_{\tb t_k^{-1}}(\mb)=0$, $\forall k=1,\ldots,n$. Then, by Proposition~\ref{prop:compactYbim}, we conclude that $N\Mb_\tb(\mb)=\sum_{k=1}^nN\Nb_{\tb t_k^{-1}}(\mb)=0$. Hence $N=0$, which ends the proof.
\end{proof}

\begin{prop}\label{prop:condexppos}
Let $E:\mb\to\ma$ be a conditional expectation from $\mb$ to $\ma$,
and $\tb\in G^n$. Let $E_\tb:\Mb_\tb(\mb)\to \Mb_\tb(\ma)$ be such
that $E_\tb(M):=(E(M_{ij}))$. Then $E_\tb$ is an idempotent and
positive $\Mb_\tb(\ma)$-bimodule map, whose image is
$\Mb_\tb(\ma)$.    
\end{prop}
\begin{proof}
It is clear that $E_\tb^2=E_\tb$ and
$E_\tb(\Mb_\tb(\mb))=\Mb_\tb(\ma)$, and it is easy to check that
$E_\tb(N_1MN_2)=N_1E_\tb(M)N_2$ $\forall N_1,N_2\in \Mb_\tb(\ma)$,
$M\in \Mb_\tb(\mb)$. Let us show that $E_\tb$ is positive. Let $u,v\in G$, $m=(m_1,\ldots,m_n)\in Y_{\tb
  u}(\mb)$ and $y=(y_1,\ldots,y_n)\in Y_{\tb v}(\ma)$, and let 
$M=\pr{m}{m}_l=(m_im_j^*)$, so $M\in \mathcal{K}(Y_{\tb
  u})^+\subseteq\Mb_\tb(\mb)^+$. Then 
$x:=m_1^*y_1+\cdots+m_n^*y_n\in B_{u^{-1}v}$. Since $E(y_j)=y_j$ and 
$E(m_j^*)=E(m_j)^*$ $\forall j$, we have  
\begin{gather*}\label{eqn:pos1}
E(x)^*E(x)
=\sum_{i,j=1}^ny_j^*E(m_j)E(m_i)^*y_i
=\pr{\big(E(m_i)E(m_j)^*\big)y}{y}_r\\
=\pr{\pr{\dot{E}_n(m)}{\dot{E}_n(m)}_ly}{y}_r, 
\end{gather*}    
where $\dot{E}_n:\mb^n\to\ma^n$ is such that
$\dot{E}_n(b_1,\ldots,b_n)=(E (b_1),\ldots,E (b_n))$. On the other hand: 
\begin{equation*}\label{eqn:pos2}
E(x^*x)
=\sum_{i,j=1}^ny_j^*E(m_jm_i^*)y_i
=\pr{(E(m_im_j^*))y}{y}_r
=\pr{E_\tb(\pr{m}{m}_l)y}{y}_r
\end{equation*}
Now, $0\leq E(x)^*E(x)\leq E(x^*x)$ by
Lemma~\ref{lem:condexp}, which in view of the computations above, 
together with Corollary~\ref{cor:Mpos}, 
yields $E_\tb(\pr{m}{m}_l)-\pr{\dot{E}_n(m)}{\dot{E}_n(m)}_l\in
\Mb_\tb(\ma)^+$. In particular we have that
$E_\tb(\pr{m}{m}_l)\in\Mb_\tb(\ma)^+$. Thus
$E_\tb(\sum_{k=1}^n\pr{m^{(k)}}{m^{(k)}}_l)\in \Mb_\tb(\ma)^+\subseteq
\Mb_\tb(\mb)^+$, $\forall m^{(1)},\ldots,m^{(n)}\in\cup_{u\in G}Y_{\tb u}$. It
follows that $E_\tb$ is 
positive, by Proposition~\ref{prop:compactYbim} and the well known fact
that the positive cone of a sum of ideals agrees with the sum of the
positive cones of these ideals.\end{proof}

\par As an application of the main result of the next section,
Theorem~\ref{thm:subalgebra}, we will see later, in
Proposition~\ref{prop:Etis condexp}, that $E_\tb$ is in fact a
conditional expectation from $\Mb_\tb(\mb)$ to $\Mb_\tb(\ma)$ .

\section{Conditional expectation on the full C*-algebra} 
\par If $E$ is a continuous conditional expectation from $\mb$ to
$\ma$, and $f\in C_c(\mb)$, then $E^c(f)$, defined to be
$E^c(f)(t):=E(f(t))$, is obviously an element of $C_c(\ma)$. This defines a map $E^c:C_c(\mb)\to C_c(\ma)$ which is idempotent with image $C_c(\ma)$, and also a 
$C_c(\ma)$-homomorphism of $C_c(\ma)$-bimodules.    
\par If $C^*_\diamond(\mb)$ is a C*-completion of $C_c(\mb)$ and it
turns out that $E^c$ is continuous in the norm of $C^*_\diamond(\mb)$,
then we refer to the continuous extension of $E^c$ as the unique
extension of $E$ to $C^*_\diamond(\mb)$. In what follows $\iota_K^\mb:C_K(\mb)\to C_c(\mb)$ indicates the natural inclusion.   

\begin{prop}\label{prop:ilt}
Let $E:\mb\to\ma$ be a continuous conditional expectation from $\mb$
to $\ma$, and let $E^c$ 
as defined previously. Then:
\begin{enumerate}
 \item $E^c$ is a surjective and contractive map in the
   uniform norms of $C_c(\mb)$ and $C_c(\ma)$. Moreover 
   $E^c(C_K(\mb))=C_K(\ma)$, for each compact subset $K$ of $G$.  
 \item $E^c:C_c(\mb)\to C_c(\ma)$ is continuous in the inductive
   limit topologies of these spaces. 
\end{enumerate} 
\end{prop}
\begin{proof}
By recalling, from
Corollary~\ref{cor:condexp}, that $E$ is contractive, and observing that
$C_c(\ma)\subseteq C_c(\mb)$, the first statement follows at once
from the comments preceding the proposition. To prove the second
statement is enough to see that 
$E^c\circ\iota_K^\mb:(C_K(\mb),\norm{\
}_\infty)\to(C_c(\ma),\tau_\ma)$ is continuous, for each compact
subset $K$ of $G$. But
$E^c\circ\iota_K^\mb=\iota_K^\ma\circ E^c|_{C_K(\mb)}$, 
and $E^c|_{C_K(\mb)}$ is continuous in the uniform topologies by (1). Since $\iota_K^\ma:(C_K(\ma),\norm{\ }_\infty)\to
(C_c(\ma),\tau_\ma)$ is continuous, we are done.              
\end{proof}

\par A crucial ingredient in the proof of our first main result is the
correspondence between cyclic representations and linear functionals
of positive type of a Fell bundle. For the convenience of the reader,
and in 
order to establish some notation, we recall the basic facts about
this theory developed by Fell in \cite{fdinduced} (see also
\cite[VIII.20-21]{fd}).  
\par A functional of positive type on a Fell bundle $\mb$ is a
continuous quasi-linear map $\varphi:\mb\to\C$ such that for any
finite sequence of elements 
$b_1,\ldots,b_n$ in $\mb$ we have $\sum_{i,j=1}^n\varphi(b_i^*b_j)\geq 0$. For
example, if $\pi:\mb\to B(H)$ is a representation of $\mb$ and $\xi\in
H$, the functional $\varphi_\xi:\mb\to\C$, given by
$\varphi_\xi(b)=\pr{\pi(b)\xi}{\xi}$, is a functional of positive
type. In fact, any functional of positive type can be represented in
such a way, and if one restricts to cyclic representations, then the
pair $(\pi,\xi)$ is unique up to unitary equivalence. It follows that
$\norm{\varphi}$, which is defined to be
$\norm{\varphi}:=\sup_{\norm{b}\leq 1}|\varphi(b)|$, is always finite, 
and in fact $\norm{\varphi}=\lim_\lambda\varphi(u_\lambda)$, where
$(u_\lambda)_{\lambda}$ is any approximate unit of $B_e$. In
particular $\norm{\varphi_\xi}=\norm{\xi}^2$. If
$\bar{\pi}:C^*(\mb)\to B(H)$ is the integrated form of $\pi$, then
$\bar{\varphi}_\xi:C^*(\mb)\to \C$ (the integrated form of $\varphi$)
such that $\bar{\varphi}_\xi(x)=\pr{\bar{\pi}(x)\xi}{\xi}$ is a
positive linear functional, and
$\norm{\varphi_\xi}=\norm{\bar{\varphi}_\xi}=\norm{\xi}^2$. If
$\pi^c:C_c(\mb)\to\C$ is the 
restriction of $\bar{\pi}$ to $C_c(\mb)$, then
$\varphi^c_\xi(f)=\pr{\pi^c(f)\xi}{\xi}$ is a positive linear
functional on $C_c(\mb)$, which is continuous in the inductive limit
topology of $C_c(\mb)$. Besides, any positive linear functional on
$C_c(\mb)$ that is continuous in the inductive limit topology is of
this form. Note that $\varphi_\xi^c(f)=\int_G\pr{\pi(f(t))\xi}{\xi}dt$ $\forall f\in C_c(\mb)$. 
\begin{lem}\label{lem:condexplinfunct}
Let $\varphi:\mb\to\C$ be a linear functional of positive type, and
let $\tb:=(t_1,\ldots,t_n)\in G^n$. Define
$\varphi^\tb:\Mb_\tb(\mb)\to \C$ by
$\varphi^\tb(M):=\sum_{i,j=1}^n\varphi(M_{ij})$, $\forall
M=(M_{ij})\in \Mb_\tb(\mb)$. Then $\varphi^\tb$ is a positive
linear functional on $\Mb_\tb(\mb)$.  
\end{lem}
\begin{proof}
We may suppose that $\varphi=\varphi_\xi$ for some cyclic
representation $\pi:\mb\to B(H)$ with cyclic vector $\xi$. It is clear
that $\pi^\tb:\Mb_\tb(\mb)\to M_n(B(H))\cong B(H^n)$ such that
$\pi^\tb(M)=(\pi(M_{ij}))$ is a representation and, if
$\xi^\tb:=(\xi,\ldots,\xi)\in H^n$, we have
$\varphi^\tb(M)=\pr{\pi^\tb(M)\xi^\tb}{\xi^\tb}$, so $\varphi^\tb$ is
a positive linear functional on $\Mb_\tb(\mb)$.  
\end{proof}

\begin{cor}\label{cor::condexplinfunct}
Let $E:\mb\to\ma$ be a continuous conditional expectation from $\mb$ to
$\ma$, and $\psi:\ma\to\C$ a linear functional of positive
type. Define $\varphi(b):=\psi(E(b))$, $\forall b\in \mb$. Then
$\varphi$ is a linear functional of positive type, and
$\norm{\varphi}=\norm{\psi}$.  
\end{cor}
\begin{proof}
A direct combination of Proposition~\ref{prop:condexppos} and
Lemma~\ref{lem:condexplinfunct} shows that $\varphi$ is a linear
functional of positive type, since $\sum_{i,j=1}^n\varphi(b_i^*b_j)=\psi^\tb\circ E_\tb(b_i^*b_j)\geq 0$, $\forall b_1,\ldots,b_n\in \mb$. Now, note that $(E(u_\lambda))$ is an
approximate unit of $A_e$ whenever $(u_\lambda)$ is an approximate
unit of $B_e$. Thus 
\[\norm{\varphi}=\lim_\lambda\varphi(u_\lambda)=\lim_\lambda\psi(E(u_\lambda))=\norm{\psi}.\]
\end{proof}

\begin{thm}\label{thm:subalgebra}
Let $\mb$ be a Fell bundle over the locally compact group $G$, and
suppose $E:\mb\to\ma$ is a continuous conditional expectation from
$\mb$ to its Fell subbundle $\ma$. Then $C^*(\ma)\subseteq C^*(\mb)$,
and $E$ extends uniquely to a conditional expectation
$E^u:C^*(\mb)\to C^*(\ma)$ from $C^*(\mb)$ to $C^*(\ma)$. 
\end{thm}
\begin{proof}
If $\norm{\ }_\mb$ and $\norm{\ }_\ma$ denote the norms of $C^*(\mb)$
and $C^*(\ma)$ respectively, we know that $\norm{\ }_\mb\leq\norm{\
}_\ma$ on $C_c(\ma)$. So to prove that they are equal on $C_c(\ma)$,
and therefore that $\overline{C_c(\ma)}^{\norm{\ }_\mb}=C^*(\ma)$, we
need to prove the opposite inequality. To this end, consider a 
representation $\pi:\ma\to B(H)$ such that its integrated
form $\bar{\pi}:C^*(\ma)\to B(H)$ is faithful and non-degenerate. Pick
$\xi\in H$, and let $\psi_\xi:\ma\to \C$ and $\varphi:\mb\to\C$ be
given by $\psi_\xi(a)=\pr{\pi(a)\xi}{\xi}_H$ and
$\varphi(b):=\psi_\xi(E(b))$, $\forall a\in \ma$, $b\in\mb$. According
with Corollary~\ref{cor::condexplinfunct} and the comments that
precede it, both $\psi_\xi$ and $\varphi$ are linear functionals of
positive type, and
$\norm{\psi_\xi}=\norm{\varphi}=\norm{\xi}^2$. In particular there is
a cyclic representation $\rho:\mb\to B(K)$ with cyclic vector $\eta$, 
such that $\varphi=\varphi_\eta$, that is
$\varphi(b)=\pr{\rho(b)\eta}{\eta}$, $\forall b\in B$. Note that
$\norm{\eta}=\sqrt{\norm{\varphi}}=\sqrt{\norm{\psi_\xi}}=\norm{\xi}$.  
\par Now if $f\in C_c(\mb)$ we have
$\bar{\varphi}(f)=\int_{G}\pr{\pi(E(f(t)))\xi}{\xi}dt=\pr{\bar{\pi}(E^c(f))\xi}{\xi}$. Thus: 
\begin{equation}\label{eqn:E1}
\norm{f}^2_\mb\norm{\eta}^2
\geq\norm{\bar{\rho}(f)\eta}^2
=\bar{\varphi}(f^**f)
=\pr{\bar{\pi}(E^c(f^**f))\xi}{\xi}
\geq 0,
\end{equation}
where the latter inequality holds because $\bar{\varphi}$ is a
positive linear functional. Since \eqref{eqn:E1} holds for any vector
$\xi$, and $\norm{\eta}=\norm{\xi}$, we see that
$\bar{\pi}(E^c(f^**f))$ is a positive 
operator with norm bounded by $\norm{f}^2_\mb$. Since $\bar{\pi}$ is
faithful, we conclude that $E^c(f^**f)$ is positive in
$C^*(\ma)$, and 
\begin{equation}\label{eqn:E2}
\norm{E^c(f^**f)}_\ma
=\norm{\bar{\pi}(E^c(f^**f))}\leq\norm{f^**f}_\mb.
\end{equation} 
So if $g\in C_c(\ma)$:  
$\norm{g}_\ma^2=\norm{g^**g}_\ma=\norm{E^c(g^**g)}_\ma\leq\norm{g^**g}_\mb=\norm{g}_\mb^2$. Thus
$\norm{\ }_\ma\leq\norm{\ }_\mb$ on $C_c(\mb)$, as claimed, and
therefore $C^*(\ma)\subseteq C^*(\mb)$.     
\par We next show that $E^c$ is continuous with respect to $\norm{\
}_\mb$. Let $\xi_1,\xi_2\in H$, and define
$\psi_{\xi_1+i^k\xi_2}(a):=\pr{\pi(a)(\xi_1+i^k\xi_2)}{\xi_1+i^k\xi_2}$,
$\forall a\in \ma$. Then by polarization we can write:
\[\pr{\pi(a)\xi_1}{\xi_2}
=\frac{1}{4}\sum_{k=0}^3i^k\pr{\pi(a)(\xi_1+i^k\xi_2)}{\xi_1+i^k\xi_2}
=\frac{1}{4}\sum_{k=0}^3i^k\psi_{\xi_1+i^k\xi_2}(a).\] 
If $\varphi_k(b):=\psi_{\xi_1+i^k\xi_2}(E(b))$ $\forall b\in \mb$, as
before there exist cyclic representations $(\rho_k,\eta_k)$ such that
$\varphi_k(b)=\pr{\rho_k(b)\eta_k}{\eta_k}$, and
$\norm{\eta_k}=\norm{\xi_1+i^k\xi_2}$ $\forall k=0,1,2,3$. Then, passing to the integrated representations, if $f\in C_c(\mb)$ and $\norm{\xi_1}=1=\norm{\xi_2}$ (so
$\norm{\eta_k}\leq 2$): 
\begin{gather*}
|\pr{\bar{\pi}(E^c(f))\xi_1}{\xi_2}|
=\frac{1}{4}\Big|\sum_{k=0}^3i^k\pr{\bar{\pi}(E^c(f))(\xi_1+i^k\xi_2)}{\xi_1+i^k\xi_2}\Big|\\
=\frac{1}{4}\sum_{k=0}^3i^k\psi_{\xi_1+i^k\xi_2}(E^c(f)) 
=\frac{1}{4}\Big|\sum_{k=0}^3i^k\bar{\varphi}_k(f)\Big|
=\frac{1}{4}\Big|\sum_{k=0}^3i^k\pr{\bar{\rho}(f)\eta_k}{\eta_k}\Big|\\
\leq\frac{1}{4}\sum_{k=0}^3\norm{\bar{\rho_k}(f)}\norm{\eta_k}^2
\leq\frac{1}{4}\sum_{k=0}^3\norm{f}_\mb\norm{\eta_k}^2
\leq 4\norm{f}_\mb.
\end{gather*}
It follows that 
\[\norm{E^c(f)}_\mb
=\norm{E^c(f)}_\ma
= \norm{\bar{\pi}(E^c(f))}
=\!\!\!\sup_{\norm{\xi_1},\norm{\xi_2}=1}\!|\pr{\bar{\pi}(E^c(f))\xi_1}{\xi_2}|
\leq 4\norm{f}_\mb.\]
Thus $E^c$ is continuous in $\norm{\ }_\mb$, with $\norm{E^c}\leq 4$,
so it extends by continuity to a positive  
idempotent map $E^u:C^*(\mb)\to C^*(\ma)$ whose image is clearly
$C^*(\ma)$.  
\par If $E_n:M_n(\mb)\to M_n(\ma)$ is the inflation of $E$ (see
Example~\ref{exmp:inflations}), $E_n$ is a 
continuous conditional expectation from $M_n(\mb)$ to $M_n(\ma)$. Thus
by the previous part of the proof we have $C^*(M_n(\ma))\subseteq
C^*(M_n(\mb))$, and we have a conditional expectation
$(E_n)^u:C^*(M_n(\mb))\to C^*(M_n(\ma))$ from $C^*(M_n(\mb))$ to
$C^*(M_n(\ma))$ that extends $E_n$. On the other hand we can consider
the inflation $(E^u)_n:M_n(C^*(\mb))\to M_n(C^*(\ma))$. Under the
natural identifications $C^*(M_n(\mb))\cong M_n(C^*(\mb))$, it is easy
to see that $(E^u)_n=(E^u)_n$, because they are both continuous and
they obviously agree on $C_c(M_n(\mb))\cong M_n(C_c(\mb))$. Thus each
$(E^u)_n$ is a positive map and $\norm{(E^u)_n}\leq 4$. In other
words, $E^u$ is a completely 
positive map such that $\norm{E}_{cb}\leq 4$. Since $E^u$ is
completely positive, we in fact have 
$\norm{E}=\norm{E}_{cb}=\sup_\nu\norm{E^u(w_\nu)}$, where
$(w_\nu)$ is any approximate unit of $C^*(\mb)$ (see for instance
\cite[Lemma~5.3]{l}). By \cite[VIII.5.11]{fd}, we can take an
approximate unit $(h_{V,\lambda})_{(V,\lambda)}$ of $C^*(\mb)$ of the
form $h_{V,\lambda}=\zeta_Vf_\lambda$, where $f_\lambda\in C_c(\mb)$
is such that  
$(f_\lambda(e))_\lambda$ is an approximate unit of $B_e$ and
$\norm{f_\lambda}_\infty\leq 1$ $\forall \lambda$,
and $\zeta_V\in C_c(G)$ is a non-negative continuous function with
support contained in the compact neighborhood $V$ of $e$, and such
that $\int_G\zeta_V=1$ $\forall V$.       
Therefore:    
\begin{gather}\label{eqn:end1}
\norm{E^u(\zeta_Vf_\lambda)}
=\norm{\bar{\pi}(E^c(\zeta_Vf_\lambda))}
=\sup_{\norm{\xi_1}=1,\norm{\xi_2}=1}|\pr{\bar{\pi}(E^c(\zeta_Vf_\lambda))\xi_1}{\xi_2}|
\end{gather}
We have 
\begin{gather}
|\pr{\bar{\pi}(E^c(\zeta_Vf_\lambda))\xi_1}{\xi_2}|
=\Big|\int_G\pr{\pi(E(\zeta_V(t)f_\lambda(t)))\xi_1}{\xi_2}dt\Big|\label{eqn:end2}\\    
\leq\int_G\zeta_V(t)\Big|\pr{\pi(E(f_\lambda(t)))\xi_1}{\xi_2}|\Big|dt
\leq \int_G\zeta_V(t)\norm{E(f_\lambda)}_\infty \norm{\xi_1}\,\norm{\xi_2}dt. \label{eqn:end3}   
\end{gather}
Now, since $E$ is contractive, $\norm{f_\lambda}_\infty\leq 1$ and
$\int_G\zeta_V(t)dt=1$, from \eqref{eqn:end1} and \eqref{eqn:end2}--\eqref{eqn:end3} we
conclude that
$\norm{E^u}_{cb}=\sup_{V,\lambda}\norm{E^u(\zeta_Vf_\lambda)}\leq
1$, which ends the proof. 
\end{proof}

\begin{defn}\label{defn:intform}
In the conditions of Theorem~\ref{thm:subalgebra}, we will say that
the conditional expectation $E^u$ is the integrated form of the
conditional expectation~$E$.  
\end{defn}

\begin{cor}[Amenability]\label{cor:subalgebra}
Let $\mb$ be a Fell bundle over the locally compact group $G$, and
suppose $E:\mb\to\ma$ is a continuous conditional expectation from
$\mb$ to its Fell subbundle $\ma$. If $C^*(\mb)=C_r^*(\mb)$, then also 
$C^*(\ma)=C_r^*(\ma)$. 
\end{cor}
\begin{proof}
By Theorem~\ref{thm:subalgebra}, $C^*(\ma)$ is the closure of
$C_c(\ma)$ within $C^*(\mb)$, and on the other hand $C^*_r(\ma)$ is
the closure of $C_c(\ma)$ within $C^*_r(\mb)$
(\cite[Proposition~3.2]{fa}). Since $C^*(\mb)=C_r^*(\mb)$, we then
conclude that $C^*(\ma)=C_r^*(\ma)$.   
\end{proof}


\begin{rem}\label{rem:todiscrete}
If $\mb$ is a Fell bundle over the locally compact group $G$, we denote by $G_d$ the group $G$ with the discrete topology, and by $\mb_d$ the Fell bundle over $G_d$ obtained by $\mb$ by forgetting its original topology, so $\mb_d$ is just the disjoint union of the Banach spaces $B_t$, $t\in G$. Then it is clear that if $E:\mb\to \ma$ is a continuous conditional expectation, then $E:\mb_d\to\ma_d$ is a conditional expectation as well. Therefore we also have that $C^*(\ma_d)\subseteq C^*(\mb_d)$ and $E$ defines a conditional expaectation $E_d^u:C^*(\mb_d)\to C^*(\ma_d)$.  
\end{rem}

\begin{prop}[Functoriality]\label{prop:functoriality}
Let $E:\mb\to \ma$ and $F:\mathcal{D}\to \mc$ be continuous conditional
expectations on the Fell bundles $\mb$ and $\mc$, and suppose
that $\rho:\mb\to\mathcal{D}$ is a homomorphism of Fell bundles that
intertwines $E$ and $F$: $\rho\circ E=F\circ\rho$. Then, the
induced homomorphism $\bar{\rho}:C^*(\mb)\to C^*(\mathcal{D})$ intertwines $E^u$ 
and $F^u$. 
\end{prop}
\begin{proof}
It is immediate that $\rho_c\circ E^c(f)=F^c\circ\rho_c(f)$ $\forall f\in
C_c(\mb)$, so the result follows from the continuity of the maps
$\bar{\rho}\circ E^u$ and $F^u\circ\bar{\rho}$.   
\end{proof}

\begin{prop}\label{prop:Etis condexp}
Let $E:\mb\to\ma$ be a conditional expectation from $\mb$ to $\ma$,
and $\tb\in G^n$. Let $E_\tb:\Mb_\tb(\mb)\to \Mb_\tb(\ma)$ be such
that $E_\tb(M):=(E(M_{ij}))$. Then $E_\tb$ is a continuous conditional
expectation from $\Mb_\tb(\mb)$ to $\Mb_\tb(\ma)$.
\end{prop}
\begin{proof}
The continuity of $E_\tb$ follows at once from that of $E$, since the convergence in $\Mb_\tb(\mb)$ and in $\Mb_\tb(\ma)$ is the convergence entrywise. For the rest of the proof we do not lose generality by assuming, as we will do, that the group $G$ is discrete (see Remark~\ref{rem:todiscrete}). Since $E_\tb$ is an idempotent positive map onto $\Mb_\tb(\ma)$, we only need to prove that $\norm{E_\tb}\leq 1$. 
\par Let $\pi:\ma\to B(H)$ be a non-degenerate representation of $\ma$ on the Hilbert space $H$, and $\bar{\pi}:C^*(\ma)\to B(H)$ the integrated representation of $\pi$. Therefore $\pi(a_t)=\bar{\pi}(a_t\delta_t)$ $\forall a_t\in A_t$, $t\in G$, where $a_t\delta_t\in C_c(\ma)\subseteq C^*(\ma)$ is such that $a_t\delta_t(r)=a_t$ if $r=t$,and $a_t\delta_t(r)=0$ otherwise. We will suppose that $\pi|_{A_e}$ is faithful. Then the restriction of $\pi$ to each fiber is isometric and therefore injective. Define $\pi_\tb:\Mb_\tb(\ma)\to M_n(B(H))\cong B(H^n)$ by $\pi_\tb(N)=(\pi(N_{ij}))$ (recall that $n=|\tb|$, and that $N_{ij}\in A_{t_it_j^{-1}}$ $\forall j=1,\ldots,n$). It is immediate to check that $\pi_\tb$ is a faithful representation of $\Mb_\tb(\ma)$. 
Consider now the conditional expectation $E^u:C^*(\mb)\to C^*(\ma)$ provided by Theorem~\ref{thm:subalgebra}. Then the composition $\bar{\pi}\circ E^u:C^*(\mb)\to B(H)$ is a completely positive and completely contractive map. Thus by Stinespring's theorem \cite[Theorem~1.5.3 and Remark~1.5.4]{bo} there exist a non-degenerate representation $\bar{\rho}:C^*(\mb)\to B(K)$ and a contraction $W\in B(H,K)$, such that $\bar{\pi}\circ E^u(x)=W^*\bar{\rho}(x)W$, $\forall z\in C^*(\mb)$. Let $\rho:\mb\to B(K)$ be the desintegrated representation of $\bar{\rho}$. For $M=(M_{ij})\in \Mb_\tb(\mb)$ we have:
\begin{gather*}
  \pi_\tb(E_\tb(M))=\big(\pi(E(M_{ij}))\big)=\big(\bar{\pi}(E(M_{ij})\delta_{t_it_j^{-1}})\big)
  =\big(\bar{\pi}(E^u(M_{ij}\delta_{t_it_j^{-1}})\big)\\
  =\big(W^*\bar{\rho}(M_{ij}\delta_{t_it_j^{-1}})W\big)
  =W^*_n\big(\rho(M_{ij})\big)W_n
  =W^*_n\rho_\tb(M_{ij})W_n,   
\end{gather*}
where $W_n=\textrm{diag}_{n}(W,\ldots,W)$. Then, since $\pi_\tb$ is isometric and $W_n$ and $\rho$ are contractions, we conclude that
\[\norm{E_\tb(M)}=\norm{\pi_\tb(E_\tb(M))}
  =\norm{W^*_n\rho_\tb(M_{ij})W_n}
  \leq\norm{W_n}^2\norm{\rho_\tb(M_{ij})}
  \leq\norm{M}. 
\]
Therefore $E_\tb$ is contractive, which ends the proof.
\end{proof}

\section{Conditional expectation on the reduced C*-algebra} 

\par We will see next that any continuous conditional expectation
$E:\mb\to\ma$ from the Fell bundle $\mb$ to its Fell subbundle $\ma$ can
be extended to a conditional expectation $E^r:C_r^*(\mb)\to C_r^*(\ma)$. Let us call $\mu:\ma\to\mb$ the natural inclusion. Recall that $C^*_r(\ma)$ can be
identified with the closure of the natural inclusion of
$\mu^c:C_c(\ma)\hookrightarrow C_c(\mb)$ within $C^*_r(\mb)$
\cite[Proposition~3.2]{fa}. However, the fact that, as in our case, we have that $C^*(\ma)\subseteq C^*(\mb)$, allows us to give a very concrete way of identifying $C^*_r(\ma)$ inside $C^*_r(\mb)$:

\begin{prop}\label{prop:identifying}
If $\ma$ is a Fell subbundle of the Fell bundle $\mb$ such that $C^*(\ma)\subseteq C^*(\mb)$, then the map $\Lambda^\ma_x\stackrel{\mu^r}{\mapsto}\Lambda^\mb_x$ from $C^*_r(\ma)$ to $C^*_r(\mb)$ is a well defined isometric homomorphism of C*-algebras.  
\end{prop}
\begin{proof}
  Let $\mb$ be a Fell bundle over the locally compact group $G$, $\pi:\mb\to B(H)$ be a represntation of $\mb$ on the Hilbert space $H$, and $\lambda:G\to B(L^2(G))$ the left regular representation of $G$. Let $\pi_\lambda:\mb\to B(L^2(G)\otimes H)$ be the representation $\pi_\lambda(b_t)=\lambda_t\otimes\pi(b_t)$. We also denote by $\pi_\lambda:C^*(\mb)\to B(L^2(G)\otimes H)$ the corresponding integrated representation. Let $C^*_R(\mb):=\pi_\lambda(C^*(\mb))\subseteq B(L^2(G)\otimes G)$. In \cite{exeng} it was proved that, in case $\pi|_{B_e}$ is faithful, then there exists an isomorphism $\Psi_\mb:C^*_R(\mb)\to C^*_r(\mb)$ such that $\Lambda_{x}^\mb=\Psi_\mb\pi_\lambda(x)$, $\forall x\in C^*(\mb)$. Let $\rho:=\pi|_\ma:\ma\to B(H)$ be the restriction of $\pi$ to the Fell subbundle $\ma$. If $\pi|_{B_e}$ is faithful, then so is $\rho|_{A_e}$, so by the above mentioned result there exists an isomorphism $\Psi_\ma:C^*_R(\ma)\to C^*_r(\ma)$ such that $\Lambda_{y}^\ma=\Psi_\ma\pi_\lambda(y)$, $\forall y\in C^*(\ma)$. Since $C^*(\ma)\subseteq C^*(\mb)$, it is clear that $\rho_\lambda=\pi_\lambda|_{C^*(\ma)}$, and the following diagram commutes:
\begin{equation}\label{eqn:euer}
  \xymatrix{
    &C^*(\ma)\ar[ld]_{\Lambda^\ma}\ar@{^{(}->}[r]\ar[d]^{\rho_\lambda}&C^*(\mb)\ar[d]_{\pi_\lambda}\ar[rd]^{\Lambda^\mb}&\\   
      C^*_r(\ma)&  C^*_R(\ma)\ar[l]_\cong^{\Psi_\ma}\ar@{^{(}->}[r]&C^*_R(\mb)\ar[r]_{\Psi_\mb}^\cong&C^*_r(\mb)
}
\end{equation}
So if $\Lambda^\ma_y=0$ for some $y\in C^*(\ma)$, then the diagram implies that also $\Lambda^\mb_y=0$. Therefore we have a well defined map $\mu^r:C^*_r(\ma)\to C^*_r(\mb)$ such that $\mu^r(\Lambda^\ma_y)=\Lambda^\mb_y$. Now it is clear that this is a homomorphism of $*$-algebras, so it is contractive. But again the diagram shows that $\Lambda^\mb_y=0$ implies $\Lambda^\ma_y=0$, so the homomorphism is isometric.   
\end{proof}

\begin{prop}\label{prop:e2}
The map $E^c:C_c(\mb)\subseteq L^2(\mb)\to C_c(\ma)\subseteq L^2(\ma)$ extends by continuity to a contractive $A_e$-linear map $E_2:L^2(\mb)\to L^2(\ma)$ 
\end{prop}
\begin{proof}
  Since $E$ is a conditional expectation into $\ma$, if $\eta\in C_c(\mb)$ and $a\in A_e$ we have: $E^c(\eta a)(t)=E(\eta(t)a)=E(\eta(t))a=E^c(\eta)(t)a=E^c(\eta)a(t)$, so $E^c(\eta a)=E^c(\eta)a$. As for the continuity, if $\eta\in C_c(\mb)$:
  \begin{gather*}
    \pr{E^c(\eta)}{E^c(\eta)}_{L^2(\ma)}
    =\int_GE^c(\eta)(r)^*E^c(\eta)(r)dr
    =\int_GE(\eta(r))^*E(\eta(r))dr\\
    =\int_GE(\eta(r)^*)E(\eta(r))dr
    \leq \int_GE(\eta(r)^*\eta(r))dr
    =E\big(\int_G\eta(r)^*\eta(r)dr\big)\\
    =E\big(\pr{\eta}{\eta}_{L^2(\mb)}\big). 
  \end{gather*}
  Then
  \begin{gather*}
    \norm{E^c(\eta)}^2_{L^2(\ma)}=\norm{\pr{E^c(\eta)}{E^c(\eta)}_{L^2(\ma)}}
    \leq \norm{E(\pr{\eta}{\eta}_{L^2(\mb)})}\\
    \leq\norm{\pr{\eta}{\eta}_{L^2(\mb)}}
    =\norm{\eta}^2_{L^2(\mb)}.
  \end{gather*} 
  \end{proof}

  \par Consider now the map $[\,,\,]:L^2(\mb)\times L^2(\mb)\to A_e$, given by $[\xi,\eta]:=E(\pr{\xi}{\eta}_{L^2(\mb)})$. Then $[\,,\,]$ is a semi-inner product on the right $A_e$-module $L^2(\mb)$, so $\xi\mapsto \norm{[\xi,\xi]}^{1/2}$ is a seminorm on $L^2(\mb)$. Let $L^2_E$ be the Hausdorff completion of $L^2(\mb)$, and $P:L^2(\mb)\to L^2_E$ the canonical map. Then $L^2_E$ is a right Hilbert $A_e$-module with the inner product $\pr{\,}{}_E$ characterized by $\pr{P\xi}{P\eta}_E=[\xi,\eta]$, that is $\pr{P\xi}{P\eta}_E=E(\pr{\xi}{\eta}_{L^2(\mb)})$ (see \cite[pages 3,4]{l}).

\begin{prop}\label{prop:fie}
Let $T\in\mathcal{L}(L^2(\mb))$, Then $T$ induces an adjointable operator $T^E\in \mathcal{L}(L^2_E)$, characterized by $T^EP\xi=PT\xi$, $\forall \xi\in L^2(\mb)$. The map $\varphi_E:\mathcal{L}(L^2(\mb))\to\mathcal{L}(L^2_E)$ thus defined, $T\mapsto T^E$, is a unital homomorphism of C*-algebras. 
\end{prop}
\begin{proof}
  Given $\xi\in L^2(\mb)$, we have $T\xi\in L^2(\mb)$. Now, since $E$ is increasing, and using \cite[Proposition~1.2]{l}, we have:  
  \begin{gather*}
    \pr{PT\xi}{PT\xi}_E=E(\pr{T\xi}{T\xi}_{L^2(\mb)})
    =E(\pr{T^*T\xi}{\xi}_{L^2(\mb)})\\
    \leq \norm{T}^2E(\pr{\xi}{\xi}_{L^2(\mb)})
    =\norm{T}^2\pr{P\xi}{P\xi}_E
  \end{gather*}
  which shows that $P\xi\mapsto PT\xi$ is a bounded map with norm at most equal to $\norm{T}$, so it extends to a bounded map $T^E:L^2_E\to L^2_E$. Moreover, if $\xi,\eta\in L^2(\mb)$:
  \begin{gather*}
    \pr{T^EP\xi}{P\eta}_E=\pr{PT\xi}{P\eta}_E=E(\pr{T\xi}{\eta}_{L^2(\mb)})\\
    =E(\pr{\xi}{T^*\eta}_{L^2(\mb)})=\pr{P\xi}{PT^*\eta}_{E}
    =\pr{P\xi}{(T^*)^EP\eta}_{E}.
  \end{gather*}
  Thus $T^E$ is adjointable and $(T^E)^*=(T^*)^E$. The rest of the proof is routine and left to the reader. 
\end{proof}

\begin{prop}\label{prop:isomV}
Let $V:L^2(\ma)\to L^2_E$ be given by $V\xi:=P\xi$. Then $V$ is an adjointable isometry, whose adjoint $V^*$ is determined by $V^*P\eta=E^c\eta$, $\forall \eta\in C_c(\mb)$.
\end{prop}
\begin{proof}
  If $\xi_1,\xi_2\in L^2(A)$:
$    \pr{V\xi_1}{V\xi_2}_E=\pr{P\xi_1}{P\xi_2}_E=E(\pr{\xi_1}{\xi_2}_{L^2(\mb)})
\\=E(\pr{\xi_1}{\xi_2}_{L^2(\ma)})=\pr{\xi_1}{\xi_2}_{L^2(\ma)},$ which shows that $V$ is an isometry.
\par Now if $\eta\in C_c(\mb)$, as shown along the proof of Proposition~\ref{prop:e2}, we have $\pr{E^c(\eta)}{E^c(\eta)}_{L^2(\ma)}
  \leq E\big(\pr{\eta}{\eta}_{L^2(\mb)}\big)$.

Thus the map $P\eta\mapsto E^c\eta$ from $P(C_c(\mb))$ to $C_c(\ma)$ is well defined and is contractive, so it extends to a contraction $L^2_E\to L^2(\ma)$. Finally, if $\xi\in C_c(\ma)$, $\eta\in C_c(\mb)$:
\begin{gather*}
  \pr{V\xi}{P\eta}_E=\pr{P\xi}{P\eta}_E=E(\pr{\xi}{\eta}_{L^2(\mb)})
  =E\big(\int_G\xi(r)^*\eta(r)dr\big)\\
  =\int_GE\big(\xi(r)^*\eta(r)\big)dr
  =\int_G\xi(r)^*E\big(\eta(r)\big)dr
  =\int_G\xi(r)^*E^c(\eta)(r)dr\\
  =\pr{\xi}{E^c(\eta)}_{L^2(\ma)}
  \end{gather*}
Consequently $V$ is adjointable, and $V^*(P\eta)=E^c(\eta)$ $\forall \eta\in C_c(\mb)$. 
\end{proof}
\medskip
\par The isometry $V$ is useful to better understand the structure of the Hilbert $A_e$-module $L^2_E$. For if $\mathsf{e}:=VV^*\in\mathcal(L^2_E)$, then $\mathsf{e}^*\mathsf{e}=(VV^*)(VV^*)=V(V^*V)V^*=VV^*=\mathsf{e}$, so $\mathsf{e}$ is a projection. Thus we can decompose $L^2_E=Y\oplus Y^\perp$ as Hilbert $A_e$-modules, where $Y=\mathsf{e}(L^2_E)=V(L^2(\ma))$ (and therefore $Y^\perp=(\mathsf{id}-\mathsf{e}_(L^2_E)$). Then $\mathcal{L}(L^2_E)\cong \begin{pmatrix}\mathcal{L}(Y)&\mathcal{L}(Y^\perp,Y)\\\mathcal{L}(Y,Y^\perp)&\mathcal{L}(Y^\perp)\end{pmatrix}$.\par By means of the isometry $V$ we still define one more map that is important to our purposes, namely $Ad_V:\mathcal{L}(L^2_E)\to \mathcal{L}(L^2(\ma))$, such that $Ad_V(S):=V^*SV$.
\begin{lem}\label{lem:last}
  Let $x\in C^*(\mb)$. Then we have $Ad_V((\Lambda^\mb_x)^E)=\Lambda^\ma_{E^u(x)}$ $\forall x\in C^*(\mb)$, that is, the following diagram is commutative:
\[
\xymatrix{
  L^2(\ma)\ar[r]^{\ \ V}\ar[d]_{\Lambda^\ma_{E^u(x)}}&L^2_E\ar[d]^{(\Lambda^\mb_x)^E}\\
  L^2(\ma)&L^2_E\ar[l]^{\ \ V^*}
  }
 \]  
\end{lem}
\begin{proof}
  Given $\xi\in C_c(\ma)\subseteq L^2(\ma)$ and $f\in C_c(\mb)\subseteq C^*(\mb)$: 
  \[V^*(\Lambda^\mb_f)^EV\xi=V^*P\Lambda^\mb_f\xi=E^c(f*\xi)=E^c(f)*\xi=\Lambda^\ma_{E^c(f)}(\xi).\]
  Then $V^*(\Lambda^\mb_f)^EV=\Lambda^\ma_{E^c(f)}$ $\forall f\in C_c(\mb)$. Since $C_c(\mb)$ is dense in $C^*(\mb)$ and $E^u$ is the continuous extension of $E^c$ to $C^*(\mb)$, we conclude that the continuous maps $x\mapsto V^*(\Lambda^\mb_x)^EV$ and $x\mapsto \Lambda^\ma_{E^u(x)}$ agree. 
  \end{proof}


Define $\phi:\mathcal{L}(L^2(\mb))\to\mathcal{L}(L^2(\ma))$  
  by $\phi=Ad_V\circ\varphi_E$ (Proposition~\ref{prop:fie} and Lemma~\ref{lem:last}). According to Lemma~\ref{lem:last} we have $\phi(\Lambda^\mb_x)=\Lambda^\ma_{E^u(x)}$ $\forall x\in C^*(\mb)$, so we have $\phi(C^*_r(\mb))=C^*_r(\ma)$. 

\begin{prop}\label{prop:subalgebrared}
The map $\phi$ is a positive contraction, such that
$\phi(\Lambda_x^\mb)=\Lambda_{E^u(x)}^\ma$ for all $x\in C^*(\mb)$. Besides 
$\phi(C_r^*(\mb))=C_r^*(\ma)$ and $\phi\mu^r=Id_{C_r^*(\ma)}$ (the map $\mu^r$ was defined in Proposition~\ref{prop:identifying}).
\end{prop}
\begin{proof} 
  Since $V$ is an isometry, $Ad_V$ is a contraction, and is clearly positive as well. Also $\varphi_E$ is a positive contraction, being a homomorphism of C*-algebras. Therefore $\phi=Ad_V\circ\varphi_E$ is also positive and contractive. The fact that $\phi(\Lambda_x^\mb)=\Lambda_{E^u(x)}^\ma$ is just a rephrasing of Lemma~\ref{lem:last}, from where it follows also that $\phi(C_r^*(\mb))=C_r^*(\ma)$. Finally, if $y\in C^*(\ma)$, we have $\phi(\mu^r(\Lambda_y^\ma))=\phi(\Lambda_y^\mb)=\Lambda^\ma_{E^u(y)}=\Lambda_y^\ma$, which shows that $\phi\mu^r=Id_{C_r^*(\ma)}$. 
\end{proof}
\medskip
\par We arrive now to our second main result of this work: up to naturally identifying $C^*_r(\ma)$ inside $C^*_r(\mb)$, we can extend the conditional expectation $E:\mb\to\ma$ to a conditional expectation $E^r:C^*_r(\mb)\to C^*_r(\ma)$ between their reduced cross-sectional algebras.

\begin{thm}\label{thm:subalgebrared}
Let $E:\mb\to\ma$ be a continuous conditional expectation from the
Fell bundle $\mb$ to its Fell subbundle $\ma$, and let $C:=\mu^r(C^*_r(\ma))\subseteq C^*_r(\mb)$ (Proposition~\ref{prop:identifying}). Then $E$ can be
extended to a conditional expectation $E^r:C^*_r(\mb)\to C$. More precisely, the map given by $E^r(\Lambda_x^\mb):=\mu^r(\phi(\Lambda_{E^u(x)}^\mb))$, $\forall \Lambda_x^\mb\in C^*_r(\mb)$, is a conditional expectation from $C^*_r(\mb)$ to $C$.  
\end{thm}
\begin{proof}
  The map $E^r$ is a positive contraction, because both $\mu^r$ and $\phi$ so are. Moreover $E^r(C^*_r(\mb))=\mu^r(\phi(C^*_r(\mb)))=\mu^r(C^*_r(\ma))=C$.
  Let $x\in C^*(\mb)$ and $y\in C^*(\ma)$, so $\Lambda_x^\mb\in C^*_r(\mb)$ and $\Lambda_y^\mb\in C$. We have

  \begin{gather*}
    E^r(\Lambda_x^\mb\Lambda_y^\mb)=E^r(\Lambda_{xy}^\mb)
    =\Lambda_{E^u(xy)}^\mb 
    =\Lambda_{E^u(x)y}^\mb
    =\Lambda_{E^u(x)}^\mb\Lambda_y^\mb
    =E^r(\Lambda_{x}^\mb)\Lambda_y^\mb.
  \end{gather*}
  Similarly, we have $  E^r(\Lambda_y^\mb\Lambda_x^\mb)=\Lambda_y^\mb E^r(\Lambda_{x}^\mb)$. Consequently, $E^r$ is a $C$-bimodule map.  
  Finally, since $\phi\mu^r=Id_{C_r^*(\ma)}$, for $\Lambda_x^\mb\in C^*_r(\mb)$ we have: 
\begin{gather*}
  (E^r)^2(\Lambda_x^\mb)=E^r\big(\mu^r(\phi(\Lambda_{E^u(x)}^\mb))\big)
  =\mu^r\Big(\phi\big(\mu^r(\phi(\Lambda_{E^u(x)}^\mb)) \big)\Big)\\
  =\mu^r\big(\phi(\Lambda_{E^u(x)}^\mb)\big)
  =E^r(\Lambda_x^\mb).
\end{gather*}
Then $(E^r)^2=E^r$. We have shown that $E^r$ is an idempotent, positive, and contractive $C$-bimodule map whose image is $C$. That is, $E^r$ is a conditional expectation.    
\end{proof}


\begin{thebibliography}{100}

\bibitem{fa} F. Abadie, \textit{Enveloping actions and Takai
    duality for partial actions}, \textit{J. Funct. Anal.}
  \textbf{197} (2003), 14--67.
  
\bibitem{trings} F.~Abadie, D.~Ferraro, \textit{Applications of
    ternary rings to C*-algebras,} Adv. Operator Theory \textbf{2}
  (2017), 293-317 (electronic).
  
\bibitem{equivfb} F.~Abadie, D.~Ferraro, \textit{Equivalence of Fell
    bundles over groups}, Journal of Operator Theory, \textbf{81}
  no. 2, (2019), 273--319.  


\bibitem{fatens} F.~Abadie, \textit{Tensor products of Fell bundles over groups}, preprint sent to publication, \texttt{https://arxiv.org/ps/funct-an/9712006}, 2024. 
  
\bibitem{ad1} C.~Anantharaman-Delaroche, \textit{Action moyennable d'un groupe localement compact sur une alg\`ebre de von Neumann}, Math. Scand. \textbf{45} (1979), 289-304.

\bibitem{ad2} C.~Anantharaman-Delaroche, \textit{Sur la moyennabilit\'e des actions libres d'un groupe localement compact dans une alg\`ebre de
von Neumann}, C. R. Acad. Sei. Paris Ser. A-B \textbf{289} (1979), 605-607.

\bibitem{bo} N.~P.~Brown, N.~Ozawa, \textit{C*-algebras and finite-dimensional approximations}, Graduate Studies in Mathematics, \textbf{88}, Amer. Math. Soc., Providence, RI, 2008.  


\bibitem{examen} R.~Exel, \textit{Amenability for Fell Bundles}, J. Reine
    Angew. Math. {\bf 492} (1997), 41-43. 

\bibitem{exeng} R.~Exel, C.~-K.~Ng, \textit{Approximation property of $C^*$-algebraic bundles}, Math. Proc. Cambridge Philos. Soc. {\bf 132} (2002), no.~3, 509--522. 

\bibitem{ruybook} R.~Exel, \textit{Partial Dynamical Systems, Fell
    Bundles and Applications}, Mathematical Surveys and Monographs
  \textbf{224}, AMS, 2017. 

\bibitem{fdinduced} J.~M.~Fell, \textit{Induced representations and
    Banach *-Algebraic Bundles}, Lecture Notes in Mathematics,
  \textbf{582}. Springer-Verlag, Berlin-New York, 1977.  

\bibitem{fd} J.~M.~Fell, R.~S.~Doran, \textit{Representations of *-algebras,
    locally compact groups, and Banach *-algebraic bundles}, Pure and Applied
    Mathematics vol. {\bf 125} and {\bf 126}, Academic Press, 1988. 


\bibitem{ired} S.~Itoh, \textit{Reduced $C^*$-crossed products and conditional expectations}, Bull. Kyushu Inst. Tech. Math. Natur. Sci. No. \textbf{29} (1982), 1–8.
\bibitem{i} S.~Itoh, \textit{Conditional expectations in $C^*$-crossed
  products}, Trans. Amer. Math. Soc. \textbf{267} (2) (1981), 661-667.   

\bibitem{kad} R.~.V~.Kadison, \textit{Non-commutative Conditional Expectations and their Applications}, Contemporary Mathematics Volume \textbf{365}, 2004.

\bibitem{mk} M.~Khoshkam, \textit{Hilbert C*-modules and conditional
    expectations on crossed products}, J. Austral. Math. Soc. (Series
  A) \textbf{61} (1996), 106-118 


\bibitem{l} E.~C.~Lance, {\em Hilbert $C^*$-modules. A toolkit for
    operator algebraists,}, London Mathematical Society, 
    Lecture Note Series {\bf 210}, Cambridge University Press, 1995.   


\bibitem{suzuki} Y.~Suzuki, \textit{Simple equivariant C*-algebras whose full and reduced crossed products coincide}, J. Noncommut. Geom. \textbf{13} (2019), no. 4, 1577–1585.



\end{thebibliography}
\end{document}